\documentclass[12pt]{amsart}
\usepackage{amsmath, amsthm, amscd, amsfonts, amssymb, graphicx, color}
\usepackage[pagebackref=false,colorlinks,linkcolor=blue,citecolor=red]{hyperref}
\usepackage[all]{xy}

\usepackage{pgf,tikz,float}
\usepackage{mathrsfs}
\usetikzlibrary{arrows}

\def\NZQ{\mathbb}               
\def\NN{{\NZQ N}}

\def\ZZ{{\NZQ Z}}

\def\frk{\frak}               

\def\Phi{{\frk n}}
\def\Phi{{\frk N}}
\def\KK{{\mathbb K}}
\def\opn#1#2{\def#1{\operatorname{#2}}} 
\opn\chara{char}
\opn\length{\ell}
\opn\pd{pd}
\opn\rk{rk}
\opn\projdim{proj\,dim}
\opn\injdim{inj\,dim}
\opn\rank{rank}
\opn\depth{depth}
\opn\grade{grade}
\opn\height{height}
\opn\embdim{emb\,dim}
\opn\codim{codim}

\opn\Tr{Tr}
\opn\bigrank{big\,rank}
\opn\superheight{superheight}
\opn\lcm{lcm}
\opn\trdeg{tr\,deg}
\opn\reg{reg}
\opn\hilb{Hilb}
\opn\hpolynomial{h}
\opn\cdeg{cdeg}
\opn\lreg{lreg}
\opn\ini{in}
\opn\lpd{lpd}
\opn\size{size}
\opn\bigsize{bigsize}
\opn\cosize{cosize}
\opn\bigcosize{bigcosize}
\opn\sdepth{sdepth}
\opn\sreg{sreg}
\opn\link{link}
\opn\fdepth{fdepth}
\opn\lin{lin}
\opn\ini{in}
\opn\div{div}
\opn\Div{Div}
\opn\cl{cl}
\opn\Cl{Cl}
\opn\Spec{Spec}
\opn\Supp{Supp}
\opn\supp{supp}
\opn\Sing{Sing}
\opn\Ass{Ass}
\opn\Min{Min}
\opn\Mon{Mon}
\opn\dstab{dstab}
\opn\astab{astab}
\opn\Syz{Syz}
\opn\Ann{Ann}
\opn\Rad{Rad}
\opn\Soc{Soc}
\opn\Im{Im}
\opn\Ker{Ker}
\opn\Coker{Coker}
\opn\Am{Am}
\opn\Hom{Hom}
\opn\Tor{Tor}
\opn\Ext{Ext}
\opn\End{End}
\opn\Aut{Aut}
\opn\id{id}

\opn\nat{nat}
\opn\pff{pf}
\opn\Pf{Pf}
\opn\GL{GL}
\opn\SL{SL}
\opn\mod{mod}
\opn\ord{ord}
\opn\Gin{Gin}
\opn\Hilb{Hilb}
\opn\sort{sort}
\opn\initial{init}
\opn\ende{end}
\opn\height{height}
\opn\type{type}
\opn\mdeg{mdeg}
\opn\aff{aff}
\opn\con{conv}
\opn\relint{relint}
\opn\st{st}
\opn\lk{lk}
\opn\cn{cn}
\opn\core{core}
\opn\vol{vol}
\opn\link{link}
\opn\star{star}
\opn\lex{lex}
\opn\sign{sign}
\opn\gr{gr}

\def\pot#1#2{#1[\kern-0.28ex[#2]\kern-0.28ex]}

\opn\dirlim{\underrightarrow{\lim}}
\opn\inivlim{\underleftarrow{\lim}}
\let\to=\rightarrow

\def\Implies{\ifmmode\Longrightarrow \else
	\unskip${}\Longrightarrow{}$\ignorespaces\fi}
\def\implies{\ifmmode\Rightarrow \else
	\unskip${}\Rightarrow{}$\ignorespaces\fi}
\def\iff{\ifmmode\Longleftrightarrow \else
	\unskip${}\Longleftrightarrow{}$\ignorespaces\fi}

\let\:=\colon
\newtheorem{Theorem}{Theorem}[section]
\newtheorem{Lemma}[Theorem]{Lemma}
\newtheorem{Corollary}[Theorem]{Corollary}
\newtheorem{Proposition}[Theorem]{Proposition}

\newtheorem{Definition}[Theorem]{Definition}

\newtheorem{Conjecture}[Theorem]{Conjecture}

\let\epsilon\varepsilon
\let\kappa=\varkappa
\def\pnt{{\raise0.5mm\hbox{\large\bf.}}}

\begin{document}
	\title{Regularity of binomial edge ideals of chordal graphs}
	\author {M. Rouzbahani Malayeri, S. Saeedi Madani, D. Kiani$^\ast$}
		
	\address{Mohammad Rouzbahani Malayeri, Department of Mathematics and Computer Science, Amirkabir University of Technology (Tehran Polytechnic), Tehran, Iran}
	\email{m.malayeri@aut.ac.ir}
	
		\address{Sara Saeedi Madani, Department of Mathematics and Computer Science, Amirkabir University of Technology (Tehran Polytechnic), Tehran, Iran, and School of Mathematics, Institute for Research in Fundamental Sciences (IPM), Tehran, Iran}
	\email{sarasaeedi@aut.ac.ir}
	
		\address{Dariush Kiani, Department of Mathematics and Computer Science, Amirkabir University of Technology (Tehran Polytechnic), Tehran, Iran, and School of Mathematics, Institute for Research in Fundamental Sciences (IPM), Tehran, Iran}
	\email{dkiani@aut.ac.ir}

	\begin{abstract}
	In this paper we prove the conjectured upper bound for Castelnuovo-Mumford regularity of binomial edge ideals posed in \cite{SK1}, in the case of chordal graphs. Indeed, we show that the regularity of any chordal graph $G$ is bounded above by the number of maximal cliques of $G$, denoted by $c(G)$. Moreover, we classify all chordal graphs $G$ for which $\mathcal{L}(G)=c(G)$, where $\mathcal{L}(G)$ is the sum of the lengths of longest induced paths of connected components of $G$. We call such graphs strongly interval graphs. Moreover, we show that the regularity of a strongly interval graph $G$ coincides with $\mathcal{L}(G)$ as well as $c(G)$.
	\end{abstract}

	\thanks{$^\ast$ Corresponding author}
	
	
	\maketitle
	
	\section{Introduction}\label{introduction}
	The binomial edge ideal of a graph was introduced in $2010$ by Herzog, Hibi, Hreinsd{\'o}ttir, Kahle and Rauh in \cite{HHHKR}, and at the same time by Ohtani in \cite{O}. 
\par Let $G$ be a graph with the vertex set $V(G)=\{1,\ldots, n\}$ and the edge set $E(G)$. Let $S=\KK[x_1 ,\ldots ,x_n , y_1 , \ldots , y_n]$ be the polynomial ring over a field $\KK$. Then the \emph{binomial edge ideal} of $G$, denoted by $J_G$ is an ideal in $S$ whose generators are all quadrics of the form $f_{ij}=x_{i}y_{j}-x_{j}y_{i}$, where $\{i,j\}\in E(G)$ and $1\leq i<j\leq n$. This ideal can be seen as a determinantal ideal generated by a collection of $2$-minors of a generic $(2\times n)$-matrix whose entries are all indeterminates.
\par Many of algebraic and homological properties and invariants of $J_G$ were investigated by several authors, see e.g. \cite{EHH, EZ, HHHKR, KS1, KS2, MM, S, SK}.
\par One of the most interesting invariants arising from the minimal graded free resolution of $\dfrac{S}{J_G}$ is the  \emph{Castelnuovo-Mumford regularity}, (or regularity for simplicity), namely,
\[
\reg \dfrac{S}{J_G}=\max \{j-i: \beta_{i,j}(\dfrac{S}{J_G})\neq 0\}.
\]
\par In \cite{SK}, all graphs $G$ for which $\reg \dfrac{S}{J_G}=1$ were characterized. Later in \cite{SK2}, the authors gave a characterization of graphs $G$ in which $\reg \dfrac{S}{J_G}=2$. In \cite{MM}, Matsuda and Murai showed that $\mathcal{L}(G)\leq \reg \dfrac{S}{J_G}\leq n-1$, where $\mathcal{L}(G)$ is the sum of the lengths of longest induced paths of connected components of $G$. In addition, they conjectured that $\reg \dfrac{S}{J_G}\leq n-2$ if $G\neq P_n$ (i.e. the path on $n$ vertices). Later, in \cite{KS2}, this conjecture was proved by the second and the third authors of this paper. On the other hand, in \cite{SK} it was shown that $\reg \dfrac{S}{J_G}\leq c(G)$, for any \emph{closed} graph $G$, where $c(G)$ denotes the number of maximal cliques of $G$. Recall that $G$ is a closed graph if $J_G$ has a quadratic Gr\"obner bases. Afterwards, it was conjectured in \cite{SK1} that the latter bound also holds in general. More precisely,
\begin{Conjecture}\label{conj}\cite[page~12]{SK1}
Let $G$ be a graph. Then $\reg \dfrac{S}{J_G}\leq c(G)$.
\end{Conjecture}
Note that in the case of \emph{triangle-free} graphs, Conjecture \ref{conj} holds by the aforementioned result of Matsuda and Murai, in \cite{MM}. Indeed, if $G$ is a triangle-free graph with $r$ connected components, then the number of maximal cliques is at least $n-r$ which confirms the conjecture for $G$.
\par Recall that a \emph{chordal} graph is a graph with no induced cycle of length greater than $3$. In \cite{EZ}, Ene and Zarojanu verified Conjecture \ref{conj} for a class of chordal graphs, called \emph{block graphs} (i.e. chordal graphs in which any two maximal cliques intersect in at most one vertex). This result was slightly improved for a larger class of chordal graphs in \cite{KSN}. Very recently, in \cite{JK}, Conjecture \ref{conj} was proved for the so-called \emph{fan} graphs of complete graphs, another subclass of chordal graphs.
\par 
In this paper, we prove Conjecture \ref{conj} for all chordal graphs. Moreover, among other results, we classify all chordal graphs $G$ with $\mathcal{L}(G)=c(G)$ which gives a precise formula for this class of graphs. We call such graphs \emph{strongly interval graphs}.
\par This paper is organized as follows. In Section \ref{preliminaries}, we recall some definitions, facts and notations which are used throughout the paper. In Section \ref{main}, first we obtain an upper bound for the number of maximal cliques of chordal graphs. Then as the main result of this paper, using some lemmata, we prove Conjecture \ref{conj} for all chordal graphs in Theorem \ref{principal}. This section ends with discussing some examples of chordal graphs $G$ for which $\reg \dfrac{S}{J_G}$ attains one of $\mathcal{L}(G)$, $c(G)$, or none of them. In Section \ref{characterization}, we introduce a subclass of interval graphs and call them strongly interval graphs. Later, in Theorem \ref{strongly interval}, we classify all chordal graphs in which $\mathcal{L}(G)=c(G)$. More precisely, for a chordal graph $G$ we show that $G$ is a strongly interval graph if and only if $\mathcal{L}(G)=c(G)$. Using this classification and Theorem \ref{principal}, we get $\reg \dfrac{S}{J_G}=\mathcal{L}(G)=c(G)$, for any strongly interval graph $G$. Finally, at the end of Section \ref{characterization}, we show that there are many strongly interval graphs for which some other regularity bounds settled in \cite{KS2} and \cite{MM} could be far away from being sharp. We also discuss a very recent conjectured upper bound for $\reg \dfrac{S}{J_G}$ by Hibi and Matsuda in \cite{HM} for a class of strongly interval graphs. More precisely, for both above purposes, applying join product of graphs, we construct an infinite family $\{\mathcal{G}_t\}_{t=1}^{\infty}$ of strongly interval graphs for which 
\[
\lim_{t \to \infty}\dfrac{c(\mathcal{G}_t)}{\vert V(\mathcal{G}_t)\vert-2}=\lim_{t \to \infty}\dfrac{c(\mathcal{G}_t)}{\deg h_{\mathcal{G}_{t}}(\lambda)}=0.
\]

	\section{Preliminaries}\label{preliminaries}
In this section we recall some notions and known facts which are used throughout the paper. In this paper, all graphs are simple (i.e. with no loops, directed and multiple edges).
\par A \emph{simplicial complex} $\Delta$ on the vertex set $V=\{1,\ldots,n\}$ is a collection of subsets of $V$ with the following properties:
\begin{enumerate}
\item
for every $v\in V$, $\{v\}\in \Delta$, and
\item
$F\in \Delta$ and $E\subseteq F$ imply that $E\in \Delta$.
\end{enumerate}
The elements of $\Delta$ are called \emph{faces} of $\Delta$ and the maximal faces of $\Delta$ are called \emph{facets} of $\Delta$. We denote the set of all facets of $\Delta$ by $\mathcal{F}(\Delta)$.
\par Let $G$ be a graph. The \emph{clique complex} of $G$, denoted by $\Delta(G)$, is the simplicial complex whose faces are cliques (i.e. complete subgraphs) of $G$. We say that $v\in V(G)$ is a \emph{free vertex} of $\Delta(G)$, if $v$ belongs to exactly one maximal clique of $G$.
\par Let $G$ be a graph and $T\subseteq V(G)$. A subgraph $H$ of $G$ on the vertex set $T$ is said to be an \emph{induced subgraph} of $G$, whenever for any two vertices $v_1,v_2\in T$, one has $\{v_1,v_2\}\in E(H)$ if $\{v_1,v_2\}\in E(G)$. Now by $G-T$, we mean the induced subgraph of $G$ on the vertex set $V(G)\backslash T$. When $T=\{v\}\subseteq V(G)$, for simplicity, we use $G-v$ instead of $G-\{v\}$. A vertex $v\in V(G)$ is said to be a \emph{cut vertex} of $G$ whenever $G-v$ has more connected components than $G$. We say that $T$ has \emph{cut point property} for $G$, whenever each $v\in T$ is a cut vertex of the graph $G-(T\backslash \{v\})$. In particular the empty set $\emptyset$, has cut point property for $G$.
\par The neighborhood of a vertex $v\in V(G)$, denoted by $N_{G}(v)$, is the set of all vertices adjacent to $v$. Moreover, we set $N_{G}[v]:= N_{G}(v)\cup \{v\}$.
\par Let $G_1$ and $G_2$ be two graphs on the disjoint vertex sets $V(G_1)$ and $V(G_2)$, respectively. Then we mean by the \emph{join product} of $G_1$ and $G_2$, denoted by $G_1*G_2$, the graph on the vertex set $V(G_1)\cup V(G_2)$ and the edge set 
\[
E(G_1)\cup E(G_2)\cup \{\{u,v\}: u\in V(G_1)~\mathrm{and}~v\in V(G_2)\}.
\]
\par Let $G$ be a graph and $T\subseteq V(G)$. Suppose that $G_1 , \ldots , G_{c(T)}$ are connected components of $G-T$. Let $\tilde{G}_1 , \ldots , \tilde{G}_{c(T)}$ be complete graphs on the vertex sets $V(G_1), \ldots , V(G_{c(T)})$, respectively, and let
\[
P_T(G)=(x_v,y_v)_{v\in T} +J_{\tilde{G}_1}+\cdots +J_{\tilde{G}_{c(T)}}.
\]
Then by \cite[Theorem~3.2]{HHHKR}, it is known that $J_{G}=\bigcap\limits_{\substack{T\subseteq V(G)}}P_T(G)$. Moreover,  in \cite[Corollary~3.9]{HHHKR}, all the minimal prime ideals of $J_G$ were described. Indeed, it was shown that $P_T(G)\in \Min (J_G)$ if and only if $T$ has cut point property for $G$.

	\section{regularity of chordal graphs }\label{main}
	In this section we prove a conjecture posed in \cite{SK1}, in the case of chordal graphs, (see \cite[page~12]{SK1}). This conjecture asserts that the Castelnuovo-Mumford regularity of $\dfrac{S}{J_G}$ is bounded above by $c(G)$, for any graph $G$.  Before that, in the following proposition we give an upper bound for the number of maximal cliques of a chordal graph. This bound strengthens a bound given in \cite[Theorem~A]{CK}. This also yields that in the case of non-tree chordal graphs with no isolated vertices, the upper bound $c(G)$ is a tighter bound than the aforementioned upper bounds in \cite{KS2} and \cite{MM}. More precisely,
	\begin{Proposition}\label{pro1}
Let $G$ be a chordal graph on $n$ vertices which has a maximal clique with $t+1$ vertices for some $t\geq 1$. Then $c(G)\leq n-t$.
\end{Proposition}
\begin{proof}
Let $F$ be a maximal clique of $G$ with $\vert F \vert = t+1$. We prove the assertion by using induction on $n$. Note that if $G$ is a complete graph, then the assertion is clear. If $n=2$, then $t=1$, and hence $G$ is complete and the assertion holds. Now we may assume that $G$ is not complete. So, $\Delta(G)$ has two non-adjacent free vertices $v_1$ and $v_2$, by \cite[Theorem~2.1]{BGL}, (see also \cite{CIRS} and \cite{HHMM}). Thus $\{v_1,v_2\} \nsubseteq F$. We may assume that $v_1 \notin F$. We show that $c(G)\leq c(G-v_{1})+1$. Indeed, since $v_1$ is a free vertex of $\Delta(G)$, we have $\vert \mathcal{F}(\Delta(G-v_{1})) \vert \geq \vert \mathcal{F}(\Delta(G))\vert -1$, and hence $c(G) \leq c(G-v_1)+1$.
On the other hand, $G-v_{1}$ is chordal and $F\in \mathcal{F}(\Delta(G-v_{1}))$, since $v_1\notin F$. Therefore by induction hypothesis, $c(G-v_{1})\leq n-1-t$. So we get 
\[
c(G)\leq c(G-v_{1})+1 \leq n-1-t+1=n-t.
\]
\end{proof}

\par Now, we introduce a graph which plays an important role in our proofs. Let $G$ be a graph and $v\in V(G)$. Associated to the vertex $v$ we define a new graph, denoted by $G_{v}$, with the vertex set $V(G)$ and the edge set 
\[
E(G)\cup \{\{u,w\}: \{u,w\}\subseteq N_{G}(v)\}.
\]
\par \medskip In the following lemma, we show how the minimal prime ideals of $J_G$, $J_{G_v}$ and $J_{G-v}$ are related.
	\begin{Lemma}\label{3part}
	Let $G$ be a graph, $T\subseteq V(G)$ and $v\in V(G)$. Then the following statements hold:
	\begin{enumerate}
	\item[{(a)}] \cite[Proposition~2.1]{RR} If $v\in T$ is a free vertex of $\Delta(G)$, then $P_{T}(G)\notin \Min (J_G)$.
	\item[{(b)}] If $v\notin T$, then $P_{T}(G)=P_{T}(G_{v})$.
	\item[{(c)}] If $v\in T$, then $P_{T}(G)=(x_v, y_v)+P_{T\backslash \{v\}}(G-v)$.
	\end{enumerate}
\end{Lemma}	
\begin{proof}
(a)
For convenience of the reader we also give a short proof for this statement here. Suppose on contrary that $P_T(G)\in \Min (J_G)$.
Then, by \cite[Corollary~3.9]{HHHKR}, $T$ has cut point property for $G$. Thus $v$ is a cut vertex of the graph $G-(T\backslash \{v\})$. So $G-T$ has two connected components $H_1$ and $H_2$ where $\{v,v_1\}\in E(G)$ and $\{v,v_2\}\in E(G)$ for some vertices $v_1\in V(H_1)$ and $v_2 \in V(H_2)$. Since $v$ is a free vertex of $\Delta (G)$, it follows that $v_1$ and $v_2$ are adjacent, which contradicts the fact that $v_1$ and $v_2$ belong to different connected components of $G-T$.
\par (b)
It is clear that those connected components of $G-T$ and $G_{v}-T$ which do not contain $v$ are exactly the same. Since $v\notin T$, by definition of $G_v$ it follows that the vertex set of the connected component of $G-T$ which contains $v$ is the same as the one for $G_{v}-T$. Therefore the result follows.
\par (c)
Since $(G-v)-(T\backslash \{v\})=G-T$, the assertion is clear.
\end{proof}
\par In the following lemma, we compare the number of maximal cliques of $G$ and $G-v$, for any vertex $v$ of a graph $G$.
	\begin{Lemma}\label{general}
	Let $G$	be a graph and $v\in V(G)$. Then $c(G-v)\leq c(G)$.
	\end{Lemma}
\begin{proof}
Let  $X=\{F\backslash \{v\}: F\in \mathcal{F}(\Delta (G))\}$.  We denote by  $  X^{ ' }$ the set of all maximal elements of  $X$ with respect to inclusion. It is enough to show that $  \mathcal{F}(\Delta(G-v))\subseteq X^{ ' }$, because then we have 
\[
c(G-v)= \vert \mathcal{F}(\Delta(G-v))  \vert  \leq \vert X^{ ' }\vert \leq \vert \mathcal{F}(\Delta (G)) \vert=c(G).
\]
 Now let $E\in \mathcal{F}(\Delta(G-v))$. So, clearly  $E\in \Delta(G)$  and $v\notin E$. We show that $E\in X^{'}$. \par First, suppose that $E \cup \{v\}\in \Delta(G) $. Then $E \cup \{v\}$ is a facet of  $\Delta(G)$ and hence $E\in X$. Moreover we show that $E\in X^{'}$. Indeed, let $T\in X$ with $E\subseteq T$. Then $T=T^{'}\backslash \{v\}$ for some $T^{'} \in  \mathcal{F}(\Delta(G))$. It follows that $v\in T^{'}$, since otherwise $T=T^{'} \in \mathcal{F}(\Delta(G-v))$ and hence $T^{'}=E$. This is a contradiction, because $T^{'}$ and $E \cup \{v\}$ are both facets of $\Delta(G)$. Since $E\subseteq T^{'}\backslash \{v\}$, we have $E \cup \{v\}=T^{'}$. Therefore $E=T^{'}\backslash \{v\}=T$, so that $E\in X^{'}$. \par Next suppose that $E \cup \{v\}$ is not a face of $\Delta(G)$. Therefore $E\in \mathcal{F}(\Delta(G))$, which clearly implies that $E\in X^{'}$.
\end{proof}

\par In the next lemma, we show that for any chordal graph $G$, $c(G_v)$ is smaller than $c(G)$, for any vertex $v$ which is not a free vertex of $\Delta(G)$. Note that this does not hold for all graphs in general, see for example Figure \ref{fig1}. Nevertheless, our computational experiments by Macaulay2 show that $\reg \dfrac{S}{J_G}=4$ which confirms  Conjecture \ref{conj} for this example.

\begin{figure}[H]
\centering
\begin{tikzpicture}[scale=1.3,line cap=round,line join=round,>=triangle 45,x=1.0cm,y=1.0cm]
\draw (0.5,2.)-- (0.,1.);
\draw (-1.,0.5)-- (0.,1.);
\draw (-1.,0.5)-- (0.,0.);
\draw (0.,0.)-- (0.5,-1.);
\draw (0.5,-1.)-- (1.,0.);
\draw (1.,0.)-- (2.,0.5);
\draw (2.,0.5)-- (1.,1.);
\draw (1.,1.)-- (0.5,2.);
\draw (1.,0.)-- (1.,1.);
\draw (1.,1.)-- (0.,1.);
\draw (0.,0.)-- (0.,1.);
\draw (1.,0.)-- (0.,0.);
\begin{scriptsize}
\draw [fill=black] (0.,1.) circle (1.5pt);
\draw [fill=black] (1.,1.) circle (1.5pt);
\draw [fill=black] (1.,0.) circle (1.5pt);
\draw [fill=black] (0.,0.) circle (1.5pt);
\draw [fill=black] (0.5,2.) circle (1.5pt);
\draw [fill=black] (2.,0.5) circle (1.5pt);
\draw [fill=black] (0.5,-1.) circle (1.5pt);
\draw [fill=black] (-1.,0.5) circle (1.5pt);
\end{scriptsize}
\end{tikzpicture}
\vspace{.5cm}\caption{A non-chordal graph $G$ with $c(G_v)=c(G)=4$ for any vertex $v$ of $G$.}
\label{fig1}
\end{figure}
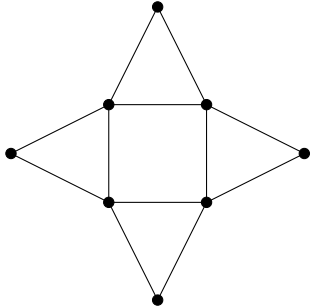

	\begin{Lemma}\label{chordal}
	Let $G$	be a chordal graph and $v$ be a vertex of $G$ which lies in $t$ maximal cliques of $G$. Then 
\[	
c(G_{v})\leq c(G)-t+1.
\]
In particular, if $t\geq2$, then $c(G_v)<c(G)$.
	\end{Lemma}
\begin{proof}
Let $c=c(G)$ and $\mathcal{F}(\Delta(G))=\{F_1, \ldots , F_c\}$. Without loss of generality assume that $F_1, \ldots , F_t$ are all facets of  $\Delta(G)$ that contain $v$. Let $Y=\{\bar{F}, F_{t+1}, \ldots , F_c\}$, where $\bar{F}$ is the clique on the vertex set $\bigcup\limits_{j=1}^{t} F_{j}$. It suffices to show that $\mathcal{F}(\Delta (G_{v}))\subseteq Y$. This then implies that
\[ 
c(G_{v})= \vert \mathcal{F}(\Delta(G_{v}))  \vert \leq \vert Y\vert = c-t+1.
\]
 \par First note that $\bar{F}\in \mathcal{F}(\Delta(G_{v}))$, since $\bar{F}=N_{G}[v]$. Now suppose there exists $F \in \mathcal{F}(\Delta(G_{v}))$ with $F\notin Y$. In particular, $F\neq \bar{F}$, so that there exists  $j \in F$ with $j \notin N_G[v]$. On the other hand, since $F\notin \{F_{t+1},\ldots, F_{c}\}$, there exist $l,k\in F$ with $l\neq k$ such that $l,k \in N_G(v)$ and  $\{l,k\}\notin E(G)$. So we get the $4$-cycle $j,l,v,k,j$ in $G$ which contradicts chordality of $G$. Therefore $\mathcal{F}(\Delta(G_{v}))\subseteq Y$.
\end{proof}
\par The following theorem is the main result of this section.
\begin{Theorem}\label{principal}
Let $G $ be a chordal graph. Then  $\reg \dfrac{S}{J_G}\leq c(G)$. 
\end{Theorem}
\begin{proof}
The idea of the proof is based on the proof of \cite[Theorem~1.1]{EHH}. We prove the assertion by using induction on $n+c(G)$, where $n= \vert V(G) \vert$. If $n+c(G)=3$, then $n=2$ and  $c(G)=1$. Therefore, $G $ is complete, and hence by \cite[Theorem~3.2]{SK}, $\reg \dfrac{S}{J_G}=1$. If $G$ is a disjoint union of complete graphs, then the result follows by \cite[Proposition~3.1.33]{V}, \cite[Lemma~2.1]{K} and  \cite[Theorem~3.2]{SK}. Otherwise $G$ has a vertex $v$ which is not a free vertex of $\Delta(G)$.\par
First we show that $G_v$ is chordal. Let $C\hspace{-1.2mm}:j_1,j_2, \ldots, j_m , j_1$ be an induced cycle of $G_{v}$ where $m\geq 4$. Then $E(C)\nsubseteq E(G)$, since $G$  is chordal. Then we may assume that $\{j_1,j_m\}\in E(C)\backslash E(G)$. Moreover, for each $t$, with $1<t<m$ we have $\{j_t,v\}\notin E(G)$, since $G$ is chordal. Therefore the cycle $C^{'}\hspace{-1.9mm}:v,j_1 ,j_2, \ldots , j_m, v$ is an induced cycle in $G$, which contradicts chordality of $G$, and hence $G_v$ is chordal.
\par Let $Q_1=\bigcap\limits_{\substack{T\subseteq V(G)\\v\notin T}}P_T(G)$. We have
\[
J_{G_{v}}=(\bigcap\limits_{\substack{T\subseteq V(G)\\v\in T}}P_T(G_{v}))\cap (\bigcap\limits_{\substack{T\subseteq V(G)\\v\notin T}}P_T(G_{v})).
\]
Since $v$ is a free vertex of $\Delta(G_v)$, it follows from Lemma \ref{3part},~part(a), that $J_{G_v}=\bigcap\limits_{\substack{T\subseteq V(G)\\v\notin T}}P_T(G_{v})$. Also by Lemma \ref{3part},~part(b), we have $\bigcap\limits_{\substack{T\subseteq V(G)\\v\notin T}}P_T(G_{v})=\bigcap\limits_{\substack{T\subseteq V(G)\\v\notin T}}P_T(G)$. Therefore $J_{G_v}=Q_1$. On the other hand, by Lemma \ref{chordal} and induction hypothesis, we have $\reg \dfrac{S}{J_{G_v}} \leq c(G_v)<c(G)$. So,
\begin{equation}\label{eq01}
\reg \dfrac{S}{Q_1}<c(G). 
\end{equation}
\par Now let $Q_2=\bigcap\limits_{\substack{T\subseteq V(G)\\v\in T}}P_T(G)$. We show that $Q_2=(x_v,y_v)+J_{G-v}$. Clearly we have $(x_v,y_v)+J_{G-v}\subseteq Q_2$. Now suppose that $f\in Q_2$. Then Lemma \ref{3part},~part(c) implies that $f\in (x_v,y_v)+ P_{T\backslash \{v\}}(G-v)$, for every $T\subseteq V(G)$ with $v\in T$. So $f=h_{T}+g_{T}$ for some $h_T \in (x_v,y_v)$ and $g_T\in P_{T\backslash \{v\}}(G-v)$. We may assume that $g_T\in S_v$, where $S_v=\KK[x_i ,y_i \mid i\in V(G)\backslash \{v\}]$. Indeed, we have $g_T=(r_{1}+r^{'}_1)g_1+ \cdots +(r_{s}+r^{'}_{s})g_s$, where $g_i=f_{l_{i}k_{i}}$ for some $l_{i}, k_{i}\in V(G)\backslash \{v\}$, $r_i \in S_v$ and $r^{'}_i \in (x_v,y_v)$ for $i=1,\ldots, s$. Let $g^{'}_T=r_{1}g_{1}+ \cdots +r_{s}g_{s}$ and $g^{''}_T=r^{'}_{1}g_1+ \cdots +r^{'}_{s}g_s$. Then $g^{'}_T \in S_v \cap P_{T-v}(G-v)$, $g^{''}_T \in (x_v,y_v)$ and $g_T= g^{'}_T + g^{''}_T$. Therefore, $ f=(h_T+g^{''}_T)+g^{'}_T$.\par
In particular, $f=h_{\{v\}}+g_{\{v\}}=h_T+g_T$, for any $T\subseteq V(G)$ with $v\in T$. Now by putting $x_v=y_v=0$ in both sides of the last equality, we get $g_{\{v\}}=g_T$. Then $g_{\{v\}}\in P_{T\backslash \{v\}}(G-v)$ for every $T\subseteq V(G)$ with $v\in T$. Hence 
\[
g_{\{v\}}\in \bigcap\limits_{\substack{T\subseteq V(G)\\v\in T}}P_{T\backslash \{v\}}(G-v)=\bigcap\limits_{\substack{W\subseteq V(G)\backslash \{v\}}}P_{W}(G-v)=J_{G-v}.
\]
So, $f\in (x_v,y_v) + J_{G-v}$, and hence $Q_2= (x_v,y_v)+ J_{G-v}$. Thus, $\reg \dfrac{S}{Q_2}=\reg \dfrac{S_v}{J_{G-v}}$. By induction hypothesis, $\reg \dfrac{S_v}{J_{G-v}}\leq c(G-v)$, since $G-v$ is chordal. This implies that $\reg \dfrac{S_v}{J_{G-v}}\leq c(G)$, since $c(G-v)\leq c(G)$, by Lemma \ref{general}. Therefore
\begin{equation}\label{eq02}
\reg \dfrac{S}{Q_2}\leq c(G).
\end{equation}
\par Next, we have $Q_1+Q_2=(x_v,y_v)+J_{G_{v}-v}$. Therefore, 
$\reg \dfrac{S}{Q_1+Q_2}=\reg \dfrac{S_v}{J_{G_{v}-v}}$. By induction hypothesis, $\reg \dfrac{S_v}{J_{G_{v}-v}} \leq c(G_{v}-v)$, since $G_{v}-v$ is chordal. Therefore, by Lemma \ref{general} and Lemma \ref{chordal}, we have
  \begin{equation}\label{eq03}
   \reg \dfrac{S}{Q_1+Q_2}< c(G).
  \end{equation} 
\par Now consider the following short exact sequence:
\[
0\longrightarrow  \dfrac{S}{J_G}\longrightarrow \dfrac{S}{Q_1}\oplus \dfrac{S}{Q_2} \longrightarrow  \dfrac{S}{Q_1+Q_2}\longrightarrow 0.
\]
By \cite[Corollary~18.7]{P}, we have
\[
 \reg \dfrac{S}{J_G} \leq \max \{ \reg \dfrac{S}{Q_1}, \reg \dfrac{S}{Q_2}, \reg\dfrac{S}{Q_1+Q_2}+1 \}.
\]
Finally by \eqref{eq01},\eqref{eq02} and \eqref{eq03}, we get
\[
\reg \dfrac{S}{J_G} \leq c(G).
\]
\end{proof}
We would like to remark that Theorem \ref{principal} together with \cite[Theorem~2.1]{SK2} also provide many non-chordal graphs for which Conjecture \ref{conj} is valid. In fact, if $G$ is the join product of two non-complete chordal graphs $G_1$ and $G_2$, then $G$ is not chordal and $\reg \dfrac{S}{J_G}\leq c(G)$, (see also \cite[page 171]{SK2}).
\par \medskip Note also that the upper bound $c(G)$ in Theorem \ref{principal} can be sharp. See for example Figure \ref{sharp}. In addition, there are some graphs $G$ where $\reg \dfrac{S}{J_G}$ attains the lower bound $\mathcal{L}(G)$ with $\mathcal{L}(G) \neq c(G)$. For instance, any \emph{caterpillar} tree, which is not a path, has this property by \cite[Theorem~4.1]{CDI}. Recall that a caterpillar tree is a tree which has a path $P$ such that any vertex of the tree has distance at most one from $P$. On the other hand, if $G$ is a tree which is not a caterpillar, then $\mathcal{L}(G)< \reg \dfrac{S}{J_G}< c(G)$ by \cite[Theorem~4.1]{CDI} and \cite[Theorem~3.2]{KS2}.

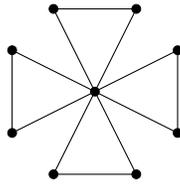
\begin{figure}[H]
\centering
\begin{tikzpicture}[scale=1.1,line cap=round,line join=round,>=triangle 45,x=1.0cm,y=1.0cm]
\draw (0.,0.)-- (0.5,1.);
\draw (-0.5,1.)-- (0.,0.);
\draw (0.,0.)-- (-1.,0.5);
\draw (0.,0.)-- (-1.,-0.5);
\draw (0.,0.)-- (-0.5,-1.);
\draw (0.,0.)-- (0.5,-1.);
\draw (0.,0.)-- (1.,-0.5);
\draw (0.,0.)-- (1.,0.5);
\draw (1.,0.5)-- (1.,-0.5);
\draw (0.5,-1.)-- (-0.5,-1.);
\draw (-1.,-0.5)-- (-1.,0.5);
\draw (-0.5,1.)-- (0.5,1.);
\begin{scriptsize}
\draw [fill=black] (0.,0.) circle (1.5pt);
\draw [fill=black] (1.,0.5) circle (1.5pt);
\draw [fill=black] (1.,-0.5) circle (1.5pt);
\draw [fill=black] (0.5,1.) circle (1.5pt);
\draw [fill=black] (-0.5,1.) circle (1.5pt);
\draw [fill=black] (-1.,0.5) circle (1.5pt);
\draw [fill=black] (-1.,-0.5) circle (1.5pt);
\draw [fill=black] (-0.5,-1.) circle (1.5pt);
\draw [fill=black] (0.5,-1.) circle (1.5pt);
\end{scriptsize}
\end{tikzpicture}
\caption{A chordal graph $G$ with $\mathcal{L}(G)=2$ and $\reg \dfrac{S}{J_G}=c(G)=4$.}
\label{sharp}
\end{figure}
	\section{strongly interval graphs}\label{characterization}
\par At the end of the previous section, we provided some examples of chordal graphs $G$ where $\reg \dfrac{S}{J_G}$ attains one of $\mathcal{L}(G)$, $c(G)$, or none of them. To complete that discussion, in this section we give a characterization of chordal graphs $G$ for which $\reg \dfrac{S}{J_G}=\mathcal{L}(G)=c(G)$.
\par The class of graphs with the above property is indeed a subclass of the well-studied so-called \emph{interval graphs.} Recall that a graph $G$ is called an interval graph if every vertex $v\in V(G)$ can be labeled with a real closed interval $I_v=[a_v, b_v]$ in such a way that two distinct vertices $v,w\in V(G)$ are adjacent if their corresponding intervals have non-empty intersection.
We identify the vertices of an interval graph $G$ with the corresponding intervals, namely we set $V(G)=\{I_1,\ldots, I_r\}$,  where $I_j=[a_j,b_j]$ with $a_j\leq b_j$ for every $1\leq j\leq r$.
\par \medskip Here we denote $\NN\cup \{0\}$ by $\NN_0$. 
\begin{Definition}\label{strongly interval}
\em{Let $k\in \NN$ and $r\in \NN_0$. Also, let $J_0=[0], J_i=[i-1,i]$ for $i=1,\ldots, k$ and $I_j=[a_j,b_j]$ such that $a_j\in \NN_0$ and $I_j\subseteq [0,k]\backslash \{k\}$, for all $j=1,\ldots, r$. Then we call the interval graph on the vertex set $\{J_i\}_{i=0}^{k}\cup \{I_j\}_{j=1}^{r}$ a \emph{connected strongly interval} graph. Moreover, we call a graph, a \emph{strongly interval} graph,  if every connected component of it is isomorphic to a connected strongly interval graph.}
\end{Definition}
Note that in the above definition, if $k=1$, then for every $r\geq 0$ we get the complete graph on ${r+2}$ vertices, and if $r=0$, then for every $k\geq 1$ we get the path $P_{k+1}$. 
\par \medskip Now we are ready to state the main theorem of this section.
\begin{Theorem}\label{strongly}
Let $G$ be a chordal graph. Then the following are equivalent:
\begin{enumerate}
\item[{(a)}] $\mathcal{L}(G)=c(G)$.
\item[{(b)}] $G$ is a strongly interval graph.
\end{enumerate}
\end{Theorem}
\begin{proof}
Without loss of generality we may assume that $G$ is connected.
\par To show $(a)\Rightarrow (b)$,  suppose that $\mathcal{L}(G)=c(G)$. For simplicity, we set $\ell=\mathcal{L}(G)$. We may assume that $P\hspace{-1.2mm}:0,1, \ldots ,\ell$ is a longest induced path in $G$ and $W=V(G)\backslash V(P)=\{u_1,\ldots, u_r\}$. If $W=\emptyset$, then the assertion is clear. So, assume that $W\neq \emptyset$.  Notice that since $P$ is an induced path in $G$ and $\ell=c(G)$, it follows that $\mathcal{F}(\Delta(G))=\{F_0,F_1,\ldots, F_{\ell -1}\}$, where for every $0\leq t \leq \ell -1$, $F_t$ is the unique maximal clique of $G$ which contains the edge $\{t,t+1\}$.
\par Let $u_i \in W$. We set $a_i=\min \{j\in V(P): \{u_i,j \} \in E(G)\}$, $d_i=\max \{j\in V(P): \{u_i,j \} \in E(G)\}$ and $I_i=[a_i,d_{i}-1+\dfrac{1}{2^{r-i+1}}]$. Also we can assume that $a_i\leq a_j$ whenever $i\leq j$. Let $J_0=[0]$ and $J_t=[t-1,t]$ for $t=1,\ldots, \ell$. We show that $G$ is an interval graph on the vertex set $\{J_t\}_{t=0}^{\ell}\cup \{I_i\}_{i=1}^{r}$, where ${I_{i}}^{,}$ s correspond to ${u_{i}}^{,}$ s, and ${J_{t}}^{,}$ s correspond to the elements of $V(P)$.  
\par Suppose that $I_i\cap I_j=\emptyset$ for some $1\leq i<j \leq r$. So $d_i\leq a_j$, since $a_i\leq a_j$. If $d_i=a_j$, then $\{u_i,u_j\}\notin E(G)$, since otherwise for the clique $\{d_i,u_i,u_j\}$, we have $\{d_i,u_i,u_j\}\subseteq F_{d_i-1}$ or $\{d_i,u_i,u_j\}\subseteq F_{d_i}$. Therefore, $\{u_j, d_{i}-1=a_j-1\}\in E(G)$ or $\{u_i,d_{i}+1\}\in E(G)$, which is impossible by the choice of $a_j$ and $d_i$. If $d_i<a_j$, then $\{u_i,u_j\}\notin E(G)$, since $G$ is chordal.
\par Next suppose that $I_i\cap I_j \neq \emptyset$ for some $1\leq i< j \leq r$. Hence $a_j\leq d_{i}-1$, since $a_{j},d_{i}\in \NN_0$. So we have $a_i\leq a_j<a_{j}+1\leq d_i$, and hence $\{u_i,a_j\} \in E(G)$ and $\{u_i,a_{j}+1\} \in E(G)$, since $G$ is chordal. On the other hand, we have $\{u_j,a_{j}+1\} \in E(G)$, since $\{u_j,a_j\}\subseteq F_{a_j}$. Therefore $\{u_i,a_j,a_j+1\}\subseteq F_{a_j}$ and $\{u_j,a_j,a_j+1\}\subseteq F_{a_j}$. This implies that $u_i$ is adjacent to $u_j$.
\par Let $1\leq i\leq r$ and $0\leq t \leq \ell$. Assume that $I_i\cap J_t=\emptyset $. Then $t<a_i$ or $t>d_i$. If $t<a_i$, then clearly $\{u_i,t\}\notin E(G)$. If $t>d_i$, then we have $\{u_i,t\}\notin E(G)$.
\par Let $I_i\cap J_t\neq \emptyset$. So we have $a_i\leq t \leq d_i$ and hence $\{u_i,t\}\in E(G)$, since $G$ is chordal.
\par Finally, it is clear that for any $i\neq t$, $J_i\cap J_t=\emptyset$ if and only if $\{i,t\}\notin E(G)$, since $P$ is an induced path in $G$. Therefore, we deduce that $G$ is a strongly interval graph, as desired.
\par \medskip Now we prove $(b)\Rightarrow (a)$. Suppose that $G$ is a strongly interval graph. So, there exist $k\in \NN$ and $r\in \NN_0$ such that $G$ is an interval graph on the vertex set $\{J_i\}_{i=0}^{k}\cup \{I_j\}_{j=1}^{r}$, where $J_0=[0], J_1=[0,1], \ldots, J_k=[k-1,k]$ and $I_j=[a_j,b_j]\subseteq [0,k]\backslash \{k\}$ for some $a_i$ and $b_j$, where $a_j\in \NN_0$ for every $j=1, \ldots, r$.
\par Let $F_i=\{J_i,J_{i+1}\}\cup \{I_t: i\in I_t\}$, for all $0\leq i \leq k-1$. We show that $\mathcal{F}(\Delta(G))=\{F_0,F_1,\ldots, F_{k-1}\}$.
\par First we have $\{F_0,F_1,\ldots, F_{k-1}\}\subseteq \mathcal{F}(\Delta(G))$. In fact, suppose on contrary that there exists $0\leq i \leq k-1 $, in which $F_i \notin \mathcal{F}(\Delta(G))$. So, there exists $F\in \mathcal{F}(\Delta(G))$ such that $F_i\subsetneqq F$. Hence, $I_t\in F\backslash F_i$, for some $1\leq t \leq r$. Also, $I_t\cap J_{i+1}\neq \emptyset $, since $G$ is an interval graph on the vertex set $\{J_i\}_{i=0}^{k}\cup \{I_j\}_{j=1}^{r}$. This implies that $a_t=i+1$, since $i\notin I_t$ and $a_t\in \NN_0$. On the other hand, $I_t\cap J_i\neq \emptyset $, namely, $[i+1,b_t]\cap J_i\neq \emptyset$, a contradiction.
\par Next we show that $\mathcal{F}(\Delta(G))\subseteq \{F_0,F_1,\ldots, F_{k-1}\}$. Suppose on contrary that there exists $F \in \mathcal{F}(\Delta(G))\backslash \{F_0,F_1,\ldots, F_{k-1}\}$. We consider the following cases:
\par First, suppose that $F\cap \{J_0,J_1,\ldots,J_k\}\neq \emptyset$. If this intersection contains two elements like $J_i$ and $J_{i+1}$, then the above argument shows that any $I_t$ in $F$ should contain $i$, and hence $F=F_i$, a contradiction. So, assume that $F\cap \{J_0,J_1,\ldots, J_k\}=\{J_j\}$ for some $j=0,\ldots,k$. Without loss of generality assume that $F=\{J_j,I_1,\ldots,I_s\}$. Then $j\notin \bigcap\limits_{\substack{1\leq i \leq s}}I_i$, since otherwise we have $F\subseteq F_j$, which is a contradiction. So, $j\notin I_t$ for some $1\leq t \leq s$. Therefore $j>b_t$, since $J_j\cap I_t\neq \emptyset$. Similarly, there exists $1\leq t^{'}\leq s$ such that $j-1\notin I_{t^{'}}$. So, $j=a_{t^{'}}$, since $J_j\cap I_{t^{'}}\neq \emptyset$ and $a_{t^{'}}\in \NN_0$. Therefore we get $a_{t^{'}}>b_t$, which contradicts the fact that $I_t\cap I_{t^{'}}\neq \emptyset$.
\par Next suppose that $F\cap \{J_0,J_1,\ldots,J_k\}=\emptyset$. Without loss of generality assume that $F=\{I_1,\ldots, I_{s}\}$ and $a_1\leq \cdots \leq a_{s}$. Let $b_m=\min\{b_i: 1\leq i \leq s\}$. Thus $a_{s}\leq b_m$, since $I_m\cap I_{s}\neq \emptyset$. Therefore for every $1\leq i\leq s$, we have $a_i\leq b_m\leq b_i$,  and hence $\lfloor b_m \rfloor \in \bigcap\limits_{\substack{1\leq i \leq s}}I_i$, because $a_i\in \NN_0$ for all $i=1,\ldots,s$. So,  $F\subseteq F_{\lfloor b_{m}\rfloor}$ which is a contradiction. Therefore $\mathcal{F}(\Delta(G))=\{F_0,F_1,\ldots,F_{k-1}\}$, which implies that $c(G)=k$.
\par Finally, $P\hspace{-1.2mm}:J_0,J_1,\ldots,J_k$ is an induced path in $G$, since $G$ is an interval graph. So, $k\leq \mathcal{L}(G)$, which together with the fact that $\mathcal{L}(G)\leq c(G)$, yields $\mathcal{L}(G)=c(G)$, as desired.
\end{proof}
It is noteworthy that according to Theorem \ref{strongly} it is crucial in the definition of strongly interval graphs to take all of the left endpoints of the intervals $I_j$ to be integers. For example, Figure \ref{fig2} shows an interval graph $G$ with $\mathcal{L}(G)=3$ and $c(G)=4$. This by the above theorem implies that $G$ is not a strongly interval graph. In this figure we see a labeling of the vertices of the graph for which all of the conditions of Definition \ref{strongly interval} is full filled, except that $[1/3,2]$ has a non-integer left endpoint.
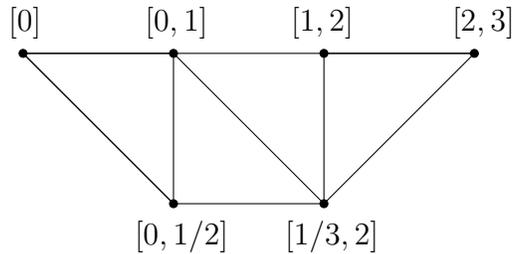
\begin{figure}[H]
\centering
\begin{tikzpicture}[scale=1,line cap=round,line join=round,>=triangle 45,x=1.0cm,y=1.0cm]
\draw  (-4.,3.)-- (-2.,3.);
\draw  (-2.,3.)-- (0.,3.);
\draw  (0.,3.)-- (2.,3.);
\draw  (-2.,1.)-- (-4.,3.);
\draw  (-2.,1.)-- (-2.,3.);
\draw  (-4.,3.)-- (-2.,3.);
\draw  (0.,3.)-- (2.,3.);
\draw  (-2.,1.)-- (-4.,3.);
\draw  (0.,1.)-- (-2.,1.);
\draw  (0.,1.)-- (2.,3.);
\draw  (0.,1.)-- (0.,3.);
\draw (-2.66,0.92) node[anchor=north west] {$[0,1/2]$};
\draw (-0.66,0.92) node[anchor=north west] {$[1/3,2]$};
\draw (-4.34,3.76) node[anchor=north west] {$[0]$};
\draw (-2.52,3.76) node[anchor=north west] {$[0,1]$};
\draw (-0.58,3.76) node[anchor=north west] {$[1,2]$};
\draw (1.56,3.76) node[anchor=north west] {$[2,3]$};
\draw  (-2.,3.)-- (0.,1.);
\begin{scriptsize}
\draw [fill=black] (-4.,3.) circle (1.5pt);
\draw [fill=black] (-2.,3.) circle (1.5pt);
\draw [fill=black] (0.,3.) circle (1.5pt);
\draw [fill=black] (2.,3.) circle (1.5pt);
\draw [fill=black] (-2.,1.) circle (1.5pt);
\draw [fill=black] (0.,1.) circle (1.5pt);
\end{scriptsize}
\end{tikzpicture}
\caption{An interval graph which is not strongly interval.}
\label{fig2}
\end{figure}
By Theorem \ref{principal} and \cite[Theorem~1.1]{MM}, we have $\mathcal{L}(G) \leq \reg \dfrac{S}{J_G}\leq c(G)$, for any chordal graph $G$. Combining this together with Theorem \ref{strongly}, we obtain the following characterization: 
\begin{Corollary}\label{reg stg}
Let $G$ be a chordal graph. Then the following are equivalent:
\begin{enumerate}
\item[{(a)}] $\reg \dfrac{S}{J_G}=\mathcal{L}(G)=c(G)$.
\item[{(b)}] $G$ is a strongly interval graph.
\end{enumerate}
\end{Corollary}
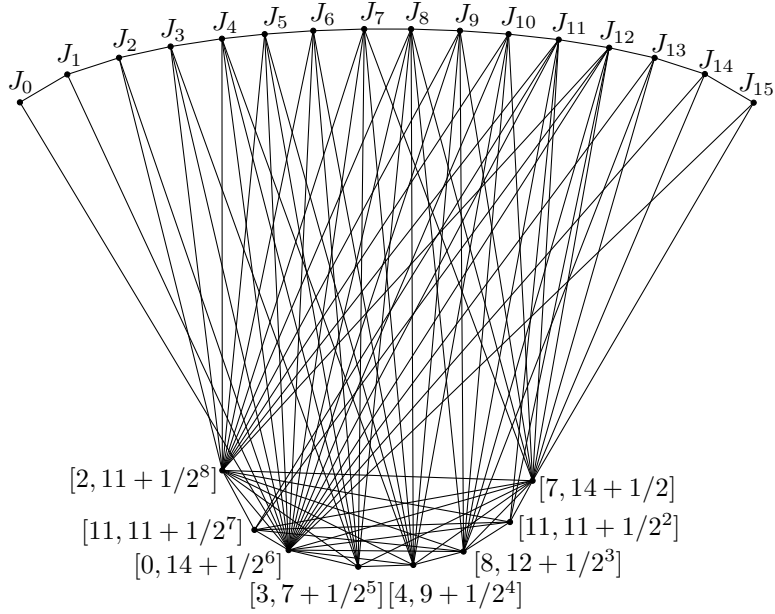
\begin{figure}[H]
\centering
\begin{tikzpicture}[scale=.65,line cap=round,line join=round,>=triangle 45,x=1.0cm,y=1.0cm]
\draw (0.,0.5345770464244869)-- (0.9690366470932892,1.1049742721388274);
\draw (0.9690366470932892,1.1049742721388274)-- (2.0271652100256032,1.447275651690137);
\draw (2.0271652100256032,1.447275651690137)-- (3.0817230921875938,1.6769181630382783);
\draw (3.0817230921875938,1.6769181630382783)-- (4.132030874991802,1.8369949698875234);
\draw (4.132030874991802,1.8369949698875234)-- (5.006171422891897,1.9305418513594415);
\draw (5.006171422891897,1.9305418513594415)-- (6.,2.);
\draw (6.,2.)-- (7.041630400335793,2.034774727928561);
\draw (7.041630400335793,2.034774727928561)-- (8.000577044186286,2.0340905250112593);
\draw (8.000577044186286,2.0340905250112593)-- (9.,2.);
\draw (9.,2.)-- (9.993828577108104,1.930541851359442);
\draw (9.993828577108104,1.930541851359442)-- (11.019644579775033,1.8173318954163666);
\draw (11.019644579775033,1.8173318954163666)-- (12.052394258301216,1.652059731750729);
\draw (12.052394258301216,1.652059731750729)-- (12.981321428113416,1.4450649002544242);
\draw (12.981321428113416,1.4450649002544242)-- (14.009420107306685,1.1136836248788047);
\draw (14.009420107306685,1.1136836248788047)-- (15.007312553252468,0.5274697728249764);
\draw (4.13975064407644,-6.998956694419681)-- (2.0271652100256032,1.447275651690137);
\draw (4.13975064407644,-6.998956694419681)-- (3.0817230921875938,1.6769181630382783);
\draw (4.13975064407644,-6.998956694419681)-- (4.132030874991802,1.8369949698875234);
\draw (5.495523599040781,-8.630339858741882)-- (0.,0.5345770464244869);
\draw (5.495523599040781,-8.630339858741882)-- (0.9690366470932892,1.1049742721388274);
\draw (5.495523599040781,-8.630339858741882)-- (2.0271652100256032,1.447275651690137);
\draw (6.919189379349907,-8.970322650891195)-- (4.132030874991802,1.8369949698875234);
\draw (6.919189379349907,-8.970322650891195)-- (5.006171422891897,1.9305418513594415);
\draw (6.919189379349907,-8.970322650891195)-- (6.,2.);
\draw (6.919189379349907,-8.970322650891195)-- (7.041630400335793,2.034774727928561);
\draw (8.045908166890984,-8.934103532210445)-- (4.132030874991802,1.8369949698875234);
\draw (8.045908166890984,-8.934103532210445)-- (5.006171422891897,1.9305418513594415);
\draw (8.045908166890984,-8.934103532210445)-- (6.,2.);
\draw (6.919189379349907,-8.970322650891195)-- (3.0817230921875938,1.6769181630382783);
\draw (9.073524590973452,-8.660182989320854)-- (8.000577044186286,2.0340905250112593);
\draw (9.073524590973452,-8.660182989320854)-- (9.,2.);
\draw (9.073524590973452,-8.660182989320854)-- (9.993828577108104,1.930541851359442);
\draw (10.4854543354431,-7.2105717254700865)-- (7.041630400335793,2.034774727928561);
\draw (10.4854543354431,-7.2105717254700865)-- (8.000577044186286,2.0340905250112593);
\draw (10.4854543354431,-7.2105717254700865)-- (9.,2.);
\draw (10.4854543354431,-7.2105717254700865)-- (9.993828577108104,1.930541851359442);
\draw (10.4854543354431,-7.2105717254700865)-- (11.019644579775033,1.8173318954163666);
\draw (10.4854543354431,-7.2105717254700865)-- (12.052394258301216,1.652059731750729);
\draw (10.4854543354431,-7.2105717254700865)-- (12.981321428113416,1.4450649002544242);
\draw (8.045908166890984,-8.934103532210445)-- (7.041630400335793,2.034774727928561);
\draw (8.045908166890984,-8.934103532210445)-- (8.000577044186286,2.0340905250112593);
\draw (8.045908166890984,-8.934103532210445)-- (9.,2.);
\draw (8.045908166890984,-8.934103532210445)-- (9.993828577108104,1.930541851359442);
\draw (9.073524590973452,-8.660182989320854)-- (11.019644579775033,1.8173318954163666);
\draw (9.073524590973452,-8.660182989320854)-- (12.052394258301216,1.652059731750729);
\draw (9.073524590973452,-8.660182989320854)-- (12.981321428113416,1.4450649002544242);
\draw (6.919189379349907,-8.970322650891195)-- (8.000577044186286,2.0340905250112593);
\draw (5.495523599040781,-8.630339858741882)-- (15.007312553252468,0.5274697728249764);
\draw (5.495523599040781,-8.630339858741882)-- (14.009420107306685,1.1136836248788047);
\draw (5.495523599040781,-8.630339858741882)-- (12.981321428113416,1.4450649002544242);
\draw (5.495523599040781,-8.630339858741882)-- (12.052394258301216,1.652059731750729);
\draw (5.495523599040781,-8.630339858741882)-- (11.019644579775033,1.8173318954163666);
\draw (5.495523599040781,-8.630339858741882)-- (9.993828577108104,1.930541851359442);
\draw (5.495523599040781,-8.630339858741882)-- (9.,2.);
\draw (5.495523599040781,-8.630339858741882)-- (8.000577044186286,2.0340905250112593);
\draw (5.495523599040781,-8.630339858741882)-- (7.041630400335793,2.034774727928561);
\draw (5.495523599040781,-8.630339858741882)-- (6.,2.);
\draw (5.495523599040781,-8.630339858741882)-- (5.006171422891897,1.9305418513594415);
\draw (5.495523599040781,-8.630339858741882)-- (4.132030874991802,1.8369949698875234);
\draw (4.13975064407644,-6.998956694419681)-- (5.006171422891897,1.9305418513594415);
\draw (5.495523599040781,-8.630339858741882)-- (4.13975064407644,-6.998956694419681);
\draw (5.495523599040781,-8.630339858741882)-- (6.919189379349907,-8.970322650891195);
\draw (5.495523599040781,-8.630339858741882)-- (8.045908166890984,-8.934103532210445);
\draw (5.495523599040781,-8.630339858741882)-- (9.073524590973452,-8.660182989320854);
\draw (5.495523599040781,-8.630339858741882)-- (10.4854543354431,-7.2105717254700865);
\draw (9.073524590973452,-8.660182989320854)-- (10.4854543354431,-7.2105717254700865);
\draw (8.045908166890984,-8.934103532210445)-- (9.073524590973452,-8.660182989320854);
\draw (10.4854543354431,-7.2105717254700865)-- (6.919189379349907,-8.970322650891195);
\draw (8.045908166890984,-8.934103532210445)-- (6.919189379349907,-8.970322650891195);
\draw (8.045908166890984,-8.934103532210445)-- (10.4854543354431,-7.2105717254700865);
\draw (6.919189379349907,-8.970322650891195)-- (4.13975064407644,-6.998956694419681);
\draw (8.045908166890984,-8.934103532210445)-- (4.13975064407644,-6.998956694419681);
\draw (-0.48,1.396666666666697) node[anchor=north west] {\footnotesize{$J_0$}};
\draw (0.52,1.9466666666666693) node[anchor=north west] {\footnotesize{$J_1$}};
\draw (1.66,2.256666666666669) node[anchor=north west] {\footnotesize{$J_2$}};
\draw (2.64,2.466666666666669) node[anchor=north west] {\footnotesize{$J_3$}};
\draw (3.7,2.61666666666669) node[anchor=north west] {\footnotesize{$J_4$}};
\draw (4.74,2.71666666666669) node[anchor=north west] {\footnotesize{$J_5$}};
\draw (5.7,2.7966666666688) node[anchor=north west] {\footnotesize{$J_6$}};
\draw (6.72,2.816666666666688) node[anchor=north west] {\footnotesize{$J_7$}};
\draw (7.64,2.81666666666669) node[anchor=north west] {\footnotesize{$J_8$}};
\draw (8.66,2.76666666666688) node[anchor=north west] {\footnotesize{$J_9$}};
\draw (9.6,2.71666666666669) node[anchor=north west] {\footnotesize{$J_{10}$}};
\draw (10.64,2.596666666666669) node[anchor=north west] {\footnotesize{$J_{11}$}};
\draw (11.62,2.416666666666669) node[anchor=north west] {\footnotesize{$J_{12}$}};
\draw (12.68,2.156666666666669) node[anchor=north west] {\footnotesize{$J_{13}$}};
\draw (13.64,1.826666666666694) node[anchor=north west] {\footnotesize{$J_{14}$}};
\draw (14.47,1.3566666666666697) node[anchor=north west] {\footnotesize{$J_{15}$}};
\draw (.77,-6.6733333333333327) node[anchor=north west]
{\footnotesize{$[2,11+1/2^8]$}};
\draw (5.495523599040781,-8.630339858741882)-- (3.0817230921875938,1.6769181630382783);
\draw (2.1,-8.43333333333326) node[anchor=north west]
{\footnotesize{$[0,14+1/2^6]$}};
\draw (4.45,-9.03333333333325) node[anchor=north west]
{\footnotesize{$[3,7+1/2^5]$}};
\draw (7.29,-9.03333333333325) node[anchor=north west]
{\footnotesize{$[4,9+1/2^4]$}};
\draw (9.05,-8.39333333333326) node[anchor=north west]
{\footnotesize{$[8,12+1/2^3]$}};
\draw (10.37,-6.893333333333327) node[anchor=north west] {\footnotesize{$[7,14+1/2]$}};
\draw (10.4854543354431,-7.2105717254700865)-- (14.009420107306685,1.1136836248788047);
\draw (4.7966652930791405,-8.216719542092415)-- (10.023248381919611,-8.05577138913707);
\draw (5.495523599040781,-8.630339858741882)-- (4.7966652930791405,-8.216719542092415);
\draw (10.023248381919611,-8.05577138913707)-- (5.495523599040781,-8.630339858741882);
\draw (10.023248381919611,-8.05577138913707)-- (11.019644579775033,1.8173318954163666);
\draw (10.023248381919611,-8.05577138913707)-- (12.052394258301216,1.652059731750729);
\draw (10.023248381919611,-8.05577138913707)-- (9.073524590973452,-8.660182989320854);
\draw (10.023248381919611,-8.05577138913707)-- (10.4854543354431,-7.2105717254700865);
\draw (4.7966652930791405,-8.216719542092415)-- (12.052394258301216,1.652059731750729);
\draw (4.7966652930791405,-8.216719542092415)-- (11.019644579775033,1.8173318954163666);
\draw (1.03,-7.733333333326) node[anchor=north west]
{\footnotesize{$[11,11+1/2^7]$}};
\draw (9.95,-7.651333333333326) node[anchor=north west]
{\footnotesize{$[11,11+1/2^2]$}};
\draw (4.13975064407644,-6.998956694419681)-- (6.,2.);
\draw (4.13975064407644,-6.998956694419681)-- (7.041630400335793,2.034774727928561);
\draw (4.13975064407644,-6.998956694419681)-- (8.000577044186286,2.0340905250112593);
\draw (4.13975064407644,-6.998956694419681)-- (9.,2.);
\draw (4.13975064407644,-6.998956694419681)-- (9.993828577108104,1.930541851359442);
\draw (4.13975064407644,-6.998956694419681)-- (11.019644579775033,1.8173318954163666);
\draw (4.13975064407644,-6.998956694419681)-- (12.052394258301216,1.652059731750729);
\draw (4.7966652930791405,-8.216719542092415)-- (4.13975064407644,-6.998956694419681);
\draw (10.023248381919611,-8.05577138913707)-- (4.13975064407644,-6.998956694419681);
\draw (4.13975064407644,-6.998956694419681)-- (10.4854543354431,-7.2105717254700865);
\draw (4.13975064407644,-6.998956694419681)-- (9.073524590973452,-8.660182989320854);
\draw (4.7966652930791405,-8.216719542092415)-- (10.4854543354431,-7.2105717254700865);
\draw (4.7966652930791405,-8.216719542092415)-- (9.073524590973452,-8.660182989320854);
\draw (10.4854543354431,-7.2105717254700865)-- (15.007312553252468,0.5274697728249764);
\begin{scriptsize}
\draw [fill=black] (0.9690366470932892,1.1049742721388274) circle (1.5pt);
\draw [fill=black] (2.0271652100256032,1.447275651690137) circle (1.5pt);
\draw [fill=black] (3.0817230921875938,1.6769181630382783) circle (1.5pt);
\draw [fill=black] (4.132030874991802,1.8369949698875234) circle (1.5pt);
\draw [fill=black] (5.006171422891897,1.9305418513594415) circle (1.5pt);
\draw [fill=black] (6.,2.) circle (1.5pt);
\draw [fill=black] (7.041630400335793,2.034774727928561) circle (1.5pt);
\draw [fill=black] (8.000577044186286,2.0340905250112593) circle (1.5pt);
\draw [fill=black] (9.,2.) circle (1.5pt);
\draw [fill=black] (9.993828577108104,1.930541851359442) circle (1.5pt);
\draw [fill=black] (11.019644579775033,1.8173318954163666) circle (1.5pt);
\draw [fill=black] (12.052394258301216,1.652059731750729) circle (1.5pt);
\draw [fill=black] (12.981321428113416,1.4450649002544242) circle (1.5pt);
\draw [fill=black] (14.009420107306685,1.1136836248788047) circle (1.5pt);
\draw [fill=black] (15.007312553252468,0.5274697728249764) circle (1.5pt);
\draw [fill=black] (0.,0.5345770464244869) circle (1.5pt);
\draw [fill=black] (4.13975064407644,-6.998956694419681) circle (1.5pt);
\draw [fill=black] (5.495523599040781,-8.630339858741882) circle (1.5pt);
\draw [fill=black] (6.919189379349907,-8.970322650891195) circle (1.5pt);
\draw [fill=black] (8.045908166890984,-8.934103532210445) circle (1.5pt);
\draw [fill=black] (9.073524590973452,-8.660182989320854) circle (1.5pt);
\draw [fill=black] (10.4854543354431,-7.2105717254700865) circle (1.5pt);
\draw [fill=black] (10.023248381919611,-8.05577138913707) circle (1.5pt);
\draw [fill=black] (4.7966652930791405,-8.216719542092415) circle (1.5pt);
\end{scriptsize}
\end{tikzpicture}
\caption{A strongly interval graph $G$ with $\reg \dfrac{S}{J_G}=\mathcal{L}(G)=c(G)=15$.}
\label{fig3}
\end{figure}
We end this section by a brief discussion on a recent conjectured upper bound by Hibi and Matsuda in \cite{HM} with a special comparison with $c(G)$ for a class of strongly interval graphs.
\par Recall that for a graph $G$ the Hilbert series of $\dfrac{S}{J_G}$ has the unique presentation of the form
\[
\hilb _{S/J_G}(\lambda)=\dfrac{h_G(\lambda)}{(1-\lambda)^d},
\]
where $h_G(\lambda)\in \ZZ[\lambda]$ and $d=\dim \dfrac{S}{J_G}$. The polynomial $h_G(\lambda)$ is called the h-polynomial of $\dfrac{S}{J_G}$. Recently in \cite{HM} it was conjectured that $\reg \dfrac{S}{J_G}\leq \deg h_G(\lambda)$, for any graph $G$. 
In the sequel, using join product of graphs, we construct an infinite family $\{\mathcal{G}_t\}_{t=1}^{\infty}$ of strongly interval graphs with
\begin{equation}\label{degree}
\lim_{t \to \infty}\dfrac{c(\mathcal{G}_t)}{\deg h_{\mathcal{G}_{t}}(\lambda)}=0.
\end{equation}
More precisely, for every $t\in \NN$, let $\mathcal{G}_t=P_m*K_t$, where $m\geq 3$ and $K_t$ is the complete graph on $t$ vertices. It is observed that $c(\mathcal{G}_t)=m-1$. On the other hand, $\dim \dfrac{S}{J_{\mathcal{G}_t}}=m+t+1$, by \cite[Proposition~4.1, part~(c)]{KS1}. Applying \cite[Theorem~4.13]{Kumar} we get
\[
\hilb_{S/J_{\mathcal{G}_t}}(\lambda)=\dfrac{(1-\lambda)^t \big((1+\lambda)^{m-1}-m\lambda +\lambda -1 \big )+(m+t-1)\lambda+1}{(1-\lambda)^{m+t+1}}.
\]
Therefore, $\deg h_{\mathcal{G}_t}(\lambda)=t+m-1$, and hence \eqref{degree} follows.
\par \medskip In addition, one can see that 
\begin{equation}\label{inf ex}
\lim_{t \to \infty}\dfrac{c(\mathcal{G}_t)}{\vert V(\mathcal{G}_t)\vert-2}=0.
\end{equation}
In fact, \eqref{degree} and \eqref{inf ex} show that for some chordal graphs $G$ on $n$ vertices the upper bound $c(G)$ could be extremely tighter than $\deg h_G(\lambda)$ and $n-2$.
\vspace{1cm}
\par \textbf{Acknowledgments:} The authors would like to thank the Institute for Research in Fundamental Sciences (IPM) for financial support. The research of the second author was in part supported by a grant from IPM (No. 96130024). The research of the third author was in part supported by a grant from IPM (No. 96050212).

\end{document}